\documentclass[12pt,oneside,reqno]{amsart}%
\makeatletter
\usepackage{amsfonts}
\usepackage{amsmath}
\usepackage{amssymb}
\usepackage{graphicx}
\usepackage{booktabs}
\usepackage{subcaption}
\usepackage{natbib}
\usepackage{epstopdf} 
\usepackage{version}
\usepackage{bm}
\usepackage{enumerate}
\usepackage{epsfig}
\usepackage{caption}
\usepackage{xcolor}
\usepackage{setspace}
\usepackage{amsthm}
\usepackage{caption}
\DeclareCaptionLabelFormat{cont}
{#1~#2\alph{ContinuedFloat}}
\usepackage{footnote}
\theoremstyle{plain}
\newtheorem{theorem}{Theorem}[section]

\newtheorem*{notation*}{Notation}
\numberwithin{equation}{section}

\newcommand{\ds}{\displaystyle}
\excludeversion{comment1}

\begin{document}

\title{The Levin Method for the Summation of One-dimensional and Multidimensional Infinite Series
}
\author{David Levin \\School of Mathematical Sciences, Tel-Aviv University, Tel-Aviv, Israel \\ levindd@gmail.com}


\begin{abstract}
The Levin method transforms the evaluation of a highly oscillatory integral into the solution of a first-order linear ODE for a slowly varying auxiliary function. This ODE is typically approximated by collocation, after which the integral value is recovered from the auxiliary function at the endpoints.

The present work develops a new extension of the Levin method for the summation of one-dimensional and multidimensional infinite oscillatory series. The summation problem is transformed into the solution of a functional equation involving transformed arguments of the unknown function. For linear phases, existence and uniqueness of a distinguished smooth solution are established under suitable regularity and end conditions. The resulting approach is particularly attractive in the multidimensional setting, where the range of existing numerical methods is relatively limited.

\end{abstract}

\maketitle

\vfill\eject

\section{Introduction}

In 1980, a PhD student asked for my advice on the efficient computation of six-dimensional oscillatory integrals. I presented him with a simple idea, which later was published in \citep{Levin1982}. The method has been further developed and analyzed in \citep{Levin1996} and in \citep{Levin1997}. Further analysis, extensions, and developments of these methods can be found in the works \citep{EvansWebster1997}, \citep{IserlesNorsett2005}, \citep{Olver2007},  \citep{HuybrechsOlver2009}, among others. Some authors refer to this approach as the Levin method; here, I will refer to it as the collocation method.

Prior to this work, together with Avram Sidi, we developed the \(d\)- and \(D\)-transformations for accelerating the convergence of infinite series and for the evaluation of infinite integrals \citep{LevinSidi1981}. These transformations were subsequently analyzed and further developed in \citep{Sidi2003}. A two-dimensional extension of the \(d\)- and \(D\)-transformations was later considered in \citep{GreifLevin1998}.

In \citep{LevinSidi1981}, the transformations are derived through an asymptotic expansion analysis of the remainders, first for infinite integrals and then for infinite series. In \citep{Levin1982}, the collocation method was originally developed for finite, highly oscillatory integrals. The same work also showed how the approach could be extended to the evaluation of infinite oscillatory integrals. The natural next step—extending the collocation approach to infinite oscillatory series—appears to have remained unexplored since then. In the present work, I show how the collocation idea introduced in \citep{Levin1982} can be adapted to the summation of one-dimensional and multidimensional infinite series.

\section{The collocation method}

\subsection{The collocation method for finite oscillatory integrals}\hfill

\medskip
Consider the oscillatory integral
\[
I=\int_a^b f(x)e^{\ds iq(x)}\,dx,
\]
where \(f\) is assumed to be smooth and slowly varying, whereas the phase
function \(q\) may induce rapid and nonuniform oscillations.

The collocation approach introduced in \citep{Levin1982} is based on the
observation that if a function \(p\) satisfies
\begin{equation}
p'(x)+iq'(x)p(x)=f(x),
\label{eq:LevinODE}
\end{equation}
then
\begin{equation}
\frac{d}{dx}\left(p(x)e^{\ds iq(x)}\right)
=
f(x)e^{\ds iq(x)},
\label{ddxp}
\end{equation}
and therefore
\begin{equation}
I
=
p(b)e^{\ds iq(b)}-p(a)e^{\ds iq(a)}.
\label{eq:LevinIntegral}
\end{equation}

The general solution of \eqref{eq:LevinODE} contains an oscillatory
homogeneous component. However, when \(f\) and \(q'\) vary slowly relative
to the oscillations of \(e^{\ds iq(x)}\), there exists a particular solution
that is itself slowly varying (see \citep{Levin1997}). The basic idea is, therefore, to approximate
such a slowly varying solution directly.

Let
\[
p_n(x)=\sum_{k=1}^n a_k u_k(x),
\]
where \(\{u_k\}_{k=1}^n\) is a prescribed family of linearly independent,
slowly varying basis functions. The coefficients \(a_k\) are determined by
imposing \eqref{eq:LevinODE} at collocation points
\[
a=x_1<x_2<\cdots<x_n=b.
\]
Thus,
\[
p_n'(x_j)+iq'(x_j)p_n(x_j)=f(x_j),
\qquad j=1,\ldots,n,
\]
or, equivalently,
\begin{equation}
\sum_{k=1}^n
a_k
\left[
u_k'(x_j)+iq'(x_j)u_k(x_j)
\right]
=
f(x_j),
\qquad j=1,\ldots,n.
\label{eq:LevinCollocation}
\end{equation}

After solving the linear system \eqref{eq:LevinCollocation}, the integral
is approximated by
\begin{equation}
I_n
=
p_n(b)e^{\ds iq(b)}-p_n(a)e^{\ds iq(a)}.
\label{eq:LevinApproximation}
\end{equation}

Thus, the direct numerical treatment of the highly oscillatory integrand
is replaced by the approximation of a slowly varying solution of the
first-order differential equation \eqref{eq:LevinODE}.

\subsection{Application to an infinite oscillatory integral}\hfill

\medskip
In \citep{Levin1982}, the applicability of the collocation idea to infinite
oscillatory integrals is demonstrated by considering the example
\begin{equation}
I
=
\int_0^\infty
\frac{e^{\ds ix}}{(x+1)^2}\,dx.
\label{eq:InfiniteExample}
\end{equation}

We transfer the problem to the finite interval case, using the transformation
\begin{equation}
x=\frac{t}{1-t},
\qquad 0\leq t<1,
\label{xtot}
\end{equation}
we have
\[
dx=\frac{dt}{(1-t)^2},
\qquad
x+1=\frac{1}{1-t},
\]
and hence
\[
\frac{dx}{(x+1)^2}=dt.
\]
The integral \eqref{eq:InfiniteExample} is therefore transformed into
\begin{equation}
I
=
\int_0^1
\exp\left(i\frac{t}{1-t}\right)\,dt.
\label{eq:TransformedInfiniteIntegral}
\end{equation}

In the notation of the finite-interval formulation,
\[
f(t)=1,
\qquad
q(t)=\frac{t}{1-t},
\]
and consequently
\[
q'(t)=\frac{1}{(1-t)^2}.
\]
The corresponding differential equation for the slowly varying function
\(p\) is
\begin{equation}
p'(t)
+
\frac{i}{(1-t)^2}p(t)
=
1.
\label{eq:InfiniteODE}
\end{equation}

Since the coefficient of \(p(t)\) becomes singular as \(t\to1\), the
equation is multiplied by \((1-t)^2\), yielding
\begin{equation}
(1-t)^2p'(t)+ip(t)=(1-t)^2.
\label{eq:RegularizedInfiniteODE}
\end{equation}

A polynomial approximation
\[
p_n(t)=\sum_{k=1}^n a_k t^{k-1}
\]
is then introduced, and the coefficients are determined by collocation:
\begin{equation}
(1-t_j)^2p_n'(t_j)+ip_n(t_j)
=
(1-t_j)^2,
\qquad j=1,\ldots,n.
\label{eq:InfiniteCollocation}
\end{equation}

Once the coefficients have been obtained, the approximation to the
integral is
\[
I_n
=
\lim_{\ds t\to 1^{-}}p_n(t)e^{\ds iq(t)}
-
p_n(0)e^{\ds iq(0)}.
\]

In the transformed formulation, the phase
$q(t)=t/(1-t)$
tends to infinity as \(t\to 1^{-}\). However, by equation \eqref{eq:RegularizedInfiniteODE}, \(p_n(1)=p(1)=0\). Consequently, the contribution from the endpoint \(t=1\) vanishes, and the approximation to the infinite integral reduces to

\[
I_n
=
-
p_n(0)e^{\ds iq(0)}.
\]

This example demonstrates that an infinite oscillatory integral can first be transformed to a finite interval, despite the infinitely rapid oscillations that arise near the endpoint \(t=1\). The resulting singularity in the derivative of the transformed phase is incorporated into the auxiliary differential equation and regularized by multiplication by a suitable vanishing factor. The regularized equation can then be treated by the same collocation principle used for finite-interval oscillatory integrals.

The numerical results reported in \citep{Levin1982} demonstrate the rapid
convergence of this procedure to the value of the original infinite
integral.

We remark that the integral in \eqref{eq:InfiniteExample} can also be efficiently approximated by the \(D\)-transformation introduced in \citep{LevinSidi1981}, and even by the earlier \(u\)-transformation of \citet{Levin1972}. The main advantage of the collocation approach becomes more pronounced, however, in the computation of multidimensional infinite oscillatory integrals.

\section{The collocation method for infinite series}\label{Sec3}

Consider an infinite series of the form
\begin{equation}
S=\sum_{n=0}^\infty f(n)e^{\ds iq(n)},
\label{sumfeq}
\end{equation}
where $f$ and $q$ are non-oscillatory real-valued functions on $\mathbb{R}_+$, with the phase function satisfying $q(n)\to\infty$ as $n\to\infty$, and $f(n)\to 0$ sufficiently rapidly to ensure the absolute convergence of the series.

Assuming that the series is absolutely convergent, we define the sequence \(\{p(n)\}_{n=0}^\infty\) by
\begin{equation}
p(n)e^{iq(n)}=
-\sum_{k=n}^{\infty} f(k)e^{iq(k)},
\qquad n\in\mathbb{N}_0,
\label{Tail}
\end{equation}
where
$$
\mathbb{N}_0=\{0,1,2,3,\ldots\}.
$$

This implies a relation replacing
the ansatz \eqref{ddxp}:
\begin{equation}
\Delta\left[p(n)e^{\ds iq(n)}\right]=f(n)e^{\ds iq(n)},\ \ \ \ n\in\mathbb{N}_0,
\label{Deltap}
\end{equation}
where $\Delta$ denotes the simple forward difference operator, which plays here the role of the derivative operator in \eqref{ddxp}.
That is, we look for a sequence $\{p(n)\}_{n=0}^\infty$ satisfying the relation \eqref{Deltap}.
Analogously to \eqref{eq:LevinIntegral}, we obtain
\begin{equation}
S=\lim_{n\to\infty}p(n)e^{\ds iq(n)}-p(0)e^{\ds iq(0)}.
\label{Svalue0}
\end{equation}
However, since the series is absolutely convergent, it follows from \eqref{Tail} that $\lim_{n\to\infty}p(n)e^{\ds iq(n)}=0$. Consequently,
\begin{equation}
S=-p(0)e^{\ds iq(0)}.
\label{Svalue}
\end{equation}

Expanding \eqref{Deltap} we obtain
\[
p(n+1)e^{\ds iq(n+1)}-p(n)e^{\ds iq(n)}=f(n)e^{\ds iq(n)},
\]
that leads to a difference equation that is the discrete analog of the ordinary differential equation \eqref{eq:LevinODE},
\begin{equation}
p(n+1)-p(n)+p(n+1)\frac{e^{\ds iq(n+1)}-e^{\ds iq(n)}}{e^{\ds iq(n)}}=f(n), \ \ \ n\in\mathbb{N}_0.
\label{Diffeq}
\end{equation}

In perfect analogy to the case of infinite integrals, we use a change of variable here
\begin{equation}
n=\frac{t_n}{1-t_n},\ \ \ t_n=\frac{n}{n+1}.
\label{ntotn}
\end{equation}

The infinite sum takes the form 

\begin{equation}
\sum_{n=0}^\infty f(\frac{t_n}{1-t_n})\exp(iq(\frac{t_n}{1-t_n})).
\label{sumfeqtn}
\end{equation}

The difference equation \eqref{Diffeq} takes the form
\begin{equation}
p(\frac{t_{n+1}}{1-t_{n+1}})r(t_n)-p(\frac{t_{n}}{1-t_{n}})=f(\frac{t_n}{1-t_n}),
\label{Diffeq2}
\end{equation}
where
\begin{equation}
r(t_n)=\exp(iq(\frac{t_{n+1}}{1-t_{n+1}}))\bigg/\exp(iq(\frac{t_n}{1-t_n})),
\label{rtn}
\end{equation}
noticing that $t_{n+1}=1/(2-t_n)$. This follows from
\begin{equation}
t_{n+1}=\frac{n+1}{n+2}=\left(\frac{t_n}{1-t_n}+1\right)\bigg/\left(\frac{t_n}{1-t_n}+2\right)=\frac{1}{2-t_n}.
\label{tnp1}
\end{equation}

The preceding developments lead to a key conceptual step in our construction. Rather than seeking the infinite sequence \(\{p(n)\}_{\ds n\in\mathbb{N}_0}\) directly, we embed it into a function \(u(t)\) defined on the finite interval \([0,1]\), requiring

\begin{equation}
p(n)=u(t_n)\equiv u(\frac{n}{n+1}), \qquad n\in\mathbb{N}_0.
\label{utop}
\end{equation}

Thus, the original problem of determining infinitely many discrete values is reformulated as the problem of approximating a single function on a finite interval.

Equation \eqref{Diffeq2} implies the relation

\begin{equation}
 u(t_{n+1})r(t_n)-u(t_n)=f\left(\frac{t_n}{1-t_n}\right).
\end{equation}
Using \eqref{tnp1}, this can be rewritten as
\begin{equation}
 u(\frac{1}{2-t_n})r(t_n)-u(t_n)=f(\frac{t_n}{1-t_n}).
 \label{Diffeq3}
\end{equation}
The next nontrivial step is to extend the relation \eqref{Diffeq3} from the discrete set of points $\{t_n\}$ to \([0,1)\), with the equation at \(t=1\) interpreted by continuity whenever the relevant limits exist.
The resulting functional equation for $u$ is
\begin{equation}
 u(\frac{1}{2-t})r(t)-u(t)=f(\frac{t}{1-t}),\ \ \ t\in [0,1),
 \label{Diffeq4}
\end{equation}
where, by \eqref{rtn},
$$r(t)=exp(iq(\frac{1}{1-t})-iq(\frac{t}{1-t})),\ \ \ t\in [0,1).$$
In particular, the quantity
$
q\left(\ds \frac{1}{1-t}\right)-q\left(\ds\frac{t}{1-t}\right)
$
has a finite limit at $t=1$ whenever
$
\lim_{x\to\infty}\bigl(q(x+1)-q(x)\bigr)
$
exists and is finite. Moreover, since \(|e^{\ds iq(n)}|=1\), equation \eqref{Tail} implies that \(p(n)\to0\) as \(n\to\infty\). Therefore, in view of \eqref{utop}, any continuous extension \(u\) satisfies
$$
u(1)=0.
$$

We note that 
$p(0)=u(0)$.
Consequently, by \eqref{Svalue},
\begin{equation}
S=-u(0)e^{\ds iq(0)}.
\label{Svalue2}
\end{equation}

\subsection{Smooth non-oscillatory solution: existence and uniqueness}\hfill

In the original method for computing highly oscillatory integrals \citep{Levin1982}, a key observation underlying the collocation strategy was the existence of a nonoscillatory solution to the ODE \eqref{eq:LevinODE}. A rigorous justification of this observation was later provided in \citep{Levin1997}.

As with the ODE \eqref{eq:LevinODE}, equation \eqref{Diffeq4} admits infinitely many solutions, which, in the present case, are generally non-smooth. However, the following theorem shows that, under mild conditions, when \(r(t)\equiv e^{i\theta}\) with \(\theta\notin 2\pi\mathbb Z\), the functional equation \eqref{Diffeq4} admits a unique smooth solution, which we regard as the distinguished non-oscillatory solution. We then approximate this distinguished solution by collocation.

\begin{theorem}
\label{ThmSmoothSolution}
Let $z=e^{i\theta},\ \ z\ne 1$ and $g\in C^\infty([0,1]), \ \ g(1)=0.$

Then the functional equation
\begin{equation}
z\,u\left(\frac{1}{2-t}\right)-u(t)=g(t),
\qquad 0\leq t<1,
\label{ConstFuncEq}
\end{equation}
has a unique solution \(u\in C^\infty([0,1])\). 
Moreover, $u(1)=0$
and
\begin{equation}
u(t)
=
-\sum_{n=0}^{\infty} z^n g(T^n(t)),
\qquad 0\leq t<1,
\label{SmoothSolutionSeries}
\end{equation}
where 
\[
T(t)=\frac{1}{2-t}.
\]

\end{theorem}
\begin{proof}
Introduce
\[
x=\frac{1}{1-t},\qquad
v(x)=u\left(1-\frac1x\right),\qquad
h(x)=g\left(1-\frac1x\right).
\]
Since
\[
T(t)=1-\frac{1}{x+1},
\]
equation \eqref{ConstFuncEq} becomes
\begin{equation}
zv(x+1)-v(x)=h(x),\qquad x\geq1.
\label{DifferenceEq}
\end{equation}

Because \(g(1)=0\) and \(g\) is smooth,
\[
h(x)=O(x^{-1}),\qquad h'(x)=O(x^{-2}).
\]
Also, since \(z\neq1\) and \(|z|=1\), the partial sums of \(z^n\)
are uniformly bounded. 
Hence, by summation by parts, the series
\[
\sum_{n=0}^\infty z^n h(x+n)
\]
converges. Therefore,
\[
v(x)=-\sum_{n=0}^\infty z^n h(x+n)
\]
is well defined and satisfies \eqref{DifferenceEq}. Returning to the
variable \(t\) gives \eqref{SmoothSolutionSeries}.

It remains to verify smoothness at \(t=1\). Since \(g\in C^\infty\)
and \(g(1)=0\), \(h\) has, for every \(N\), an asymptotic expansion
\[
h(x)
=
\frac{a_1}{x}
+\frac{a_2}{x^2}
+\cdots+
\frac{a_N}{x^N}
+O(x^{-N-1}).
\]
Substituting an expansion
\[
v(x)
=
\frac{c_1}{x}
+\frac{c_2}{x^2}
+\cdots
\]
into \eqref{DifferenceEq} determines the coefficients \(c_j\)
successively. At each step the coefficient of the new unknown
\(c_j\) is \(z-1\neq0\). Standard summation-by-parts estimates for
the remainder then show that this asymptotic expansion holds to
arbitrary order, together with its derivatives. Since
$x^{-1}=1-t$,
it follows that \(u\) extends smoothly to \(t=1\). Moreover, the
expansion begins with a multiple of \(1-t\), and hence $u(1)=0$.

Finally, suppose \(u_1\) and \(u_2\) are two continuous solutions and
set \(w=u_1-u_2\). Then
\[
z\,w(T(t))=w(t),
\]
and iteration gives
\[
w(t)=z^n w(T^n(t)).
\]
Since \(|z|=1\) and \(T^n(t)\to1\),
\[
|w(t)|=|w(1)|.
\]
Taking \(t\to1^-\) in the homogeneous equation gives
\[
(z-1)w(1)=0.
\]
Since \(z\neq1\), \(w(1)=0\), and therefore \(w\equiv0\).
Thus, the smooth solution is unique.
\end{proof}

\subsection{Collocation solution of the functional equation}\hfill

Let
\begin{equation}
u_n(x)=\sum_{k=1}^n a_k \phi_k(x),
\label{uneq}
\end{equation}
where \(\{\phi_k\}_{k=1}^n\) is a prescribed family of linearly independent,
slowly varying basis functions. Assuming that the equation extends continuously to \(t=1\), the coefficients \(a_k\) are determined by
imposing \eqref{Diffeq4} at collocation points
\[
0=x_1<x_2<\cdots<x_n=1.
\]
Thus,
\begin{equation}
 u_n(\frac{1}{2-x_j})r(x_j)-u_n(x_j)=f(\frac{x_j}{1-x_j}), \qquad j=1,\ldots,n.
 \label{DiffeqCol}
\end{equation}

\subsection{An example}
\hfill

\medskip
Consider the infinite series
\begin{equation}
S=\sum_{n=0}^\infty (n+1)^{\ds -\alpha}e^{\ds in\theta},
\label{Sexample}
\end{equation}
where \(\theta\in\mathbb{R}\) and \(\alpha>1\), the latter condition ensuring absolute convergence.

Let us derive the explicit form of the functional equation for \(u\) in this case. Since
\[
f(\frac{t}{1-t})
=
(\frac{t}{1-t}+1)^{\ds -\alpha}
=
(1-t)^{\ds \alpha},
\]
and, by \eqref{rtn}, with \(q(n)=n\theta\),
\[
r(t)
=
\frac{e^{\ds i(n+1)\theta}}{e^{\ds in\theta}}
\equiv
e^{\ds i\theta},
\]
the functional equation for \(u\) becomes
\begin{equation}
u(\frac{1}{2-t})e^{\ds i\theta}-u(t)
=
(1-t)^{\ds \alpha},
\qquad t\in[0,1].
\label{Diffeq5}
\end{equation}

For the numerical illustration, we set \(\theta=1\) and \(\alpha=2\). The corresponding value of the series, accurate to ten decimal digits, is
$$S=\sum_0^\infty \frac{e^{\ds in}}{(n+1)^2}\approx 1.0283495580+0.2750919539\, i.$$

For the collocation approximation, we employ the monomial basis

$$
\phi_k(x)=x^{k-1}, \qquad k=1,\ldots,n,
$$

together with the equidistant collocation points

$$
x_k=\frac{k-1}{n-1}, \qquad k=1,\ldots,n.
$$

By \eqref{Svalue2}, the resulting \(n\)-th order approximation to the series is

$$
S_n=-u_n(0).
$$
The results below demonstrate the high accuracy of the method. The digits that agree with the reference value are highlighted in bold.

\[
\begin{aligned}
S_{10} &= \mathbf{1.02835}51961 + \mathbf{0.27509}86725\,i,\\
S_{15} &= \mathbf{1.028349}7063 + \mathbf{0.275091}8377\,i,\\
S_{20} &= \mathbf{1.02834955}12 + \mathbf{0.27509195}28\,i,\\
S_{25} &= \mathbf{1.028349558}9 + \mathbf{0.2750919539}\,i.
\end{aligned}
\]

We remark that the condition number of the associated linear systems grows rapidly with \(n\). Nevertheless, this ill-conditioning does not appear to significantly affect the accuracy of the quantity \(u_n(0)\), which is the only value required for approximating the infinite sum.

\subsection{Rational approximations}
\label{Rational}
\hfill

\medskip
Consider the application of the above collocation procedure to the power series
\begin{equation}
S(z)=\sum_{n=0}^\infty f(n)z^n.
\label{sumfz}
\end{equation}

The resulting functional equation for this series takes the form
\begin{equation}
u(\frac{1}{2-t})z-u(t)
=
f(\frac{t}{1-t}),
\qquad t\in[0,1).
\label{Diffeq6}
\end{equation}
with the endpoint included only when the relevant limit exists.

Applying collocation to this equation with a chosen set of basis functions leads to a linear system for the coefficients of the approximate solution

$$
u_n(x)=\sum_{k=1}^n a_k\phi_k(x).
$$

\begin{equation}
 u_n(\frac{1}{2-x_j})z -u_n(x_j)=f(\frac{x_j}{1-x_j}), \qquad j=1,\ldots,n.
 \label{DiffeqColz}
\end{equation}
where $
\{x_j\}_{j=1}^n$ are the chosen collocation points. 

Each row of the resulting \(n\times n\) system matrix depends linearly on \(z\). Therefore, whenever \(A(z)\) is nonsingular, each entry of \(A(z)^{-1}\) is a rational function of \(z\):
\[
\left(A(z)^{-1}\right)_{ij}
=
\frac{p_{ij}(z)}{D(z)},
\]
where $D(z)$
is a polynomial of degree at most \(n\), and \(p_{ij}(z)\) is a polynomial of degree at most \(n-1\). It follows that the resulting approximation
$S_n(z)=-u_n(0)$ 
is a rational function of $z$.

The rational approximation \(S_n(z)\) depends on the choice of basis functions \(\{\phi_k\}\) and collocation points \(\{x_j\}\). A detailed analysis of these approximations lies beyond the scope of the present paper; nevertheless, preliminary numerical experiments reveal several interesting properties.

\section{The multidimensional case}

As a preliminary to the discussion of infinite multiple series, we first recall the application of the collocation method to finite \(d\)-dimensional oscillatory integrals.

\subsection{The collocation method for finite multidimensional oscillatory integrals}
\label{sec:multidimensional-Levin}\hfill

We consider oscillatory integrals over the unit cube
\[
    I=\int_{[0,1]^d} f(\mathbf{x})e^{\ds iq(\mathbf{x})}\,d\mathbf{x},
    \qquad
    \mathbf{x}=(x_1,\ldots,x_d),
\]
where \(f\) is assumed to be slowly varying in comparison with the
oscillatory factor \(e^{\ds iq}\).

We denote
\[
    q_j(\mathbf{x})
    =
    \frac{\partial q}{\partial x_j}(\mathbf{x}),
    \qquad j=1,\ldots,d,
\]
and
\[
    D_j=\frac{\partial}{\partial x_j}.
\]

Define the differential operator
\begin{equation}
\label{Ld-def}
    L^{(d)}p
    =
    e^{-iq(\mathbf{x})}
    D_1D_2\cdots D_d
    \left[p(\mathbf{x})e^{\ds iq(\mathbf{x})}\right].
\end{equation}

Equivalently,
\begin{equation}
\label{Ld-product}
    L^{(d)}
    =
    (D_1+iq_1)(D_2+iq_2)\cdots(D_d+iq_d),
\end{equation}
where the right-hand side denotes the  composition of differential
operators.

For example, when \(d=2\),
\[
\begin{aligned}
L^{(2)}p
&=(D_1+iq_1)(D_2+iq_2)p \\
&=p_{12}
  +iq_2p_1
  +iq_1p_2
  +(iq_{12}-q_1q_2)p.
\end{aligned}
\]

We seek a solution of
\begin{equation}
\label{d-dimensional-PDE}
    L^{(d)}p(\mathbf{x})=f(\mathbf{x}),
    \qquad \mathbf{x}\in[0,1]^d.
\end{equation}

It follows from \eqref{Ld-def} that
\[
    D_1D_2\cdots D_d
    \left[p(\mathbf{x})e^{\ds iq(\mathbf{x})}\right]
    =
    f(\mathbf{x})e^{\ds iq(\mathbf{x})}.
\]

Integrating successively with respect to
\(x_1,\ldots,x_d\), we obtain
\begin{equation}
\label{vertex-formula}
\boxed{
    I
    =
    \sum_{\boldsymbol{\epsilon}\in\{0,1\}^d}
    (-1)^{d-|\boldsymbol{\epsilon}|}
    p(\boldsymbol{\epsilon})
    e^{\ds iq(\boldsymbol{\epsilon})}
    },
\end{equation}
where
\[
    \boldsymbol{\epsilon}
    =
    (\epsilon_1,\ldots,\epsilon_d),
    \qquad
    |\boldsymbol{\epsilon}|
    =
    \epsilon_1+\cdots+\epsilon_d.
\]

We approximate \(p\) by
\begin{equation}
\label{pd-approximation}
    p_N(\mathbf{x})
    =
    \sum_{k=1}^{N} a_k u_k(\mathbf{x}),
\end{equation}
where \(u_1,\ldots,u_N\) are slowly varying basis functions.

Let
\[
    \mathbf{x}^{(1)},\ldots,\mathbf{x}^{(N)}
    \in[0,1]^d
\]
be collocation points. The coefficients are determined from
\begin{equation}
\label{d-collocation}
    L^{(d)}p_N\left(\mathbf{x}^{(j)}\right)
    =
    f\left(\mathbf{x}^{(j)}\right),
    \qquad j=1,\ldots,N.
\end{equation}

Equivalently,
\begin{equation}
\label{d-collocation-expanded}
    \sum_{k=1}^{N}
    a_k
    L^{(d)}u_k\left(\mathbf{x}^{(j)}\right)
    =
    f\left(\mathbf{x}^{(j)}\right),
    \qquad j=1,\ldots,N.
\end{equation}

After solving this linear system, the oscillatory integral is
approximated by
\begin{equation}
\label{d-Levin-quadrature}
\boxed{
    I_N
    =
    \sum_{\boldsymbol{\epsilon}\in\{0,1\}^d}
    (-1)^{d-|\boldsymbol{\epsilon}|}
    p_N(\boldsymbol{\epsilon})
    e^{\ds iq(\boldsymbol{\epsilon})}
    }.
\end{equation}

A natural choice is a tensor-product polynomial approximation
\begin{equation}
\label{tensor-polynomial}
    p_n(\mathbf{x})
    =
    \sum_{\alpha_1=0}^{n}
    \cdots
    \sum_{\alpha_d=0}^{n}
    a_{\boldsymbol{\alpha}}
    x_1^{\alpha_1}\cdots x_d^{\alpha_d},
\end{equation}
where
\[
    \boldsymbol{\alpha}
    =
    (\alpha_1,\ldots,\alpha_d).
\]

The number of unknown coefficients is
\[
    N=(n+1)^d.
\]

For example, using equidistant tensor-product collocation points,
\[
    x_j=\frac{j}{n},
    \qquad j=0,\ldots,n,
\]
the collocation set is
\[
    \left\{
    \left(
    \frac{j_1}{n},\ldots,\frac{j_d}{n}
    \right):
    0\le j_1,\ldots,j_d\le n
    \right\}.
\]

\subsection{Infinite double series}\hfill

Consider an infinite double series of the form
\begin{equation}
S=\sum_{n_1,n_2=0}^\infty f(n_1,n_2)e^{\ds iq(n_1,n_2)},
\label{sum2feq}
\end{equation}
where $f$ and $q$ are non-oscillatory functions on $\mathbb{R}_+^2$, with the
phase function satisfying \(q(x_1,x_2)\to\infty\) as \(x_1+x_2\to\infty\), and \(f(x_1,x_2)\to 0\) as \(x_1+x_2\to\infty\), sufficiently rapidly to ensure the absolute convergence of the series.

Here we define a double sequence $\{p(n_1,n_2)\}$ such that
\begin{equation}
p(n_1,n_2)e^{\ds iq(n_1,n_2)}=
\sum_{k_1=n_1}^{\infty}
\sum_{k_2=n_2}^{\infty} f(k_1,k_2)e^{\ds iq(k_1,k_2)},
\qquad n_1,n_2\in\mathbb{N}_0.
\label{Tail2}
\end{equation}

This implies the relation 
\begin{equation}
\Delta_2\Delta_1\left[p(n_1,n_2)e^{\ds iq(n_1,n_2)}\right]=f(n_1,n_2)e^{\ds iq(n_1,n_2)},\ \ \ \ n_1,n_2\in\mathbb{N}_0,
\label{Del12tap}
\end{equation}
where $\Delta_j$ is the forward difference operating on the $j$th index.
Summing up the infinite double sum
$$ S=\sum_{n1,n_2=0}^\infty \Delta_2\Delta_1\left[p(n_1,n_2)e^{\ds iq(n_1,n_2)}\right]$$
we obtain
\begin{equation}
S=c_{\infty,\infty}-c_{\infty,0}-c_{0,\infty}+p(0,0)e^{\ds iq(0,0)}
\end{equation}
where
$$
c_{\infty,\infty}=\lim_{n_1,n_2\to\infty}p(n_1,n_2)e^{\ds iq(n_1,n_2)},
\qquad
c_{\infty.0}=\lim_{n_1\to\infty}p(n_1,0)e^{\ds iq(n_1,0)},
\qquad
c_{0,\infty}=\lim_{n_2\to\infty}p(0,n_2)e^{\ds iq(0,n_2)}.
$$
Absolute convergence of the double series directly implies that all three terms at infinity vanish, and we are left with 
\begin{equation}
S=p(0,0)e^{\ds iq(0,0)}
\label{S2value}
\end{equation}

Dividing \eqref{Del12tap} by $e^{\ds iq(n_1,n_2)}$ we derive the following relation
\begin{equation}
\Delta_2\Delta_1\left[p(n_1,n_2)e^{\ds iq(n_1,n_2)}\right]e^{\ds -iq(n_1,n_2)}=f(n_1,n_2).
\label{Expand2compact}
\end{equation}

As in the univariate case considered in Section~\ref{Sec3}, we introduce the change of variables

$$
(n_1,n_2)\mapsto \bigl(s_{n_1},t_{n_2}\bigr),
$$

which maps the infinite lattice \(\mathbb{N}_0^2\) onto an infinite lattice contained in \([0,1]^2\).

\begin{equation}
n_1=\frac{s_{n_1}}{1-s_{n_1}},\ \ \ s_{n_1}=\frac{n_1}{n_1+1},\ \ \
n_2=\frac{t_{n_2}}{1-t_{n_2}},\ \ \ t_{n_2}=\frac{n_2}{n_2+1}.
\label{n2totn}
\end{equation}

Rather than seeking the infinite sequence \(\{p(n_1,n_2)\}_{n_1,n_2\in\mathbb{N}_0}\), we look for a function \(u(s,t)\) defined on \([0,1]^2\), requiring

$$
p(n_1,n_2)=u(s_{n_1},t_{n_2})\equiv u(\frac{n_1}{n_1+1},\frac{n_2}{n_2+1}), \qquad n_1,n_2\in\mathbb{N}_0.
$$

Let us also denote
$$
Q(x_1,x_2)\equiv 
e^{\ds iq(x_1,x_2)}.
$$

Equation \eqref{Expand2compact} can be rewritten as

\begin{equation}
\Delta_2\Delta_1\left[p(n_1,n_2)e^{\ds iq(n_1,n_2)}\right]\bigg/Q(\frac{s_{n_1}}{1-s_{n_1}},\frac{t_{n_2}}{1-t_{n_2}})=f(\frac{s_{n_1}}{1-s_{n_1}},\frac{t_{n_2}}{1-t_{n_2}}).
\label{u2compact}
\end{equation}

Using the relations

$$
p(n_1,n_2)=u(s_{n_1},t_{n_2}),\ \ 
s_{n_1+1}=\frac{1}{2-s_{n_1}},\ \ 
t_{n_2+1}=\frac{1}{2-t_{n_2}},\ \ n_1+1=\frac{1}{1-s_{n_1}},\ \ n_2+1=\frac{1}{1-t_{n_2}},
$$

both sides of \eqref{u2compact} can be expressed solely in terms of \(s_{n_1}\) and \(t_{n_2}\). 

The next step is to replace $s_{n_1}$ and $t_{n_2}$ by the continuous variables $s$ and $t$.
The resulting functional equation for $u(s,t)$, $(s,t)\in [0,1)^2$, is

\begin{equation}
u(\frac{1}{2-s},\frac{1}{2-t})\frac{Q(\ds \frac{1}{1-s},\frac{1}{1-t})}{Q(\ds \frac{s}{1-s},\frac{t}{1-t})}
-u(\frac{1}{2-s},t)\frac{Q(\ds \frac{1}{1-s},\frac{t}{1-t})}{Q(\ds \frac{s}{1-s},\frac{t}{1-t})}
-u(s,\frac{1}{2-t})\frac{Q(\ds \frac{s}{1-s},\frac{1}{1-t})}{Q(\ds \frac{s}{1-s},\frac{t}{1-t})}
+u(s,t)
=f(\frac{s}{1-s},\frac{t}{1-t}).
\label{Func2eq}
\end{equation}

Observe that, as $s\to1$, we have $1/(2-s)\to 1$ and $s/(1-s)\to\infty$, with analogous relations holding for $t$. Let $u(s,t)$ be a solution of \eqref{Func2eq} that extends continuously to \([0,1]^2\). 

By 
\eqref{Tail2} and absolute convergence,
$$ p(n_1,0)\to 0,\qquad p(0,n_2)\to 0,\qquad p(n_1,n_2)\to 0 $$
as the corresponding indices tend to infinity. Hence, assuming a continuous extension \(u\), it follows that
$$
u(1,1)=u(1,0)=u(0,1)=0.
$$
Moreover, $u(0,0)=p(0,0)$. Hence, by \eqref{S2value},
\begin{equation}
 S=u(0,0)e^{\ds iq(0,0)}.
 \label{S2value2}
\end{equation}
For general phase functions \(q\), the existence of a sufficiently smooth solution of \eqref{Func2eq} remains an open question.

For the important special case of a linear phase
\[
q(n_1,n_2)=n_1\theta_1+n_2\theta_2,
\qquad
e^{\ds i\theta_1}\neq1,\quad e^{\ds i\theta_2}\neq1,
\]
the existence of a smooth solution follows directly from
Theorem~\ref{ThmSmoothSolution}. 

Denoting
\[
g(s,t)=f\left(\frac{s}{1-s},\frac{t}{1-t}\right),
\]
and defining
\[
(\mathcal L_ju)(s,t)
=
e^{\ds i\theta_j}u(T(s),t)-u(s,t),
\]
\[
(\mathcal L_2u)(s,t)
=
e^{\ds i\theta_2}u(s,T(t))-u(s,t),
\]
equation \eqref{Func2eq} takes the factorized form
\[
\mathcal L_1\mathcal L_2u=g.
\]
One first applies Theorem~\ref{ThmSmoothSolution} in the \(s\)-variable,
treating \(t\) as a parameter. The resulting solution also vanishes
on \(t=1\), by uniqueness applied to the homogeneous
equation on that face. One may then apply Theorem \ref{ThmSmoothSolution} in the
\(t\)-variable.
Hence, if \(g\in C^\infty([0,1]^2)\) and
\[
g(1,t)=g(s,1)=0,
\]
successive application of Theorem~\ref{ThmSmoothSolution}
shows that there exists a unique smooth solution
\(u\in C^\infty([0,1]^2)\), satisfying
\[
u(1,t)=u(s,1)=0.
\]

\subsection{A two-dimensional
example}\label{2Dnumerical}\hfill

Consider an infinite double series of the form
\begin{equation}
S=\sum_{n_1,n_2=0}^\infty (n_1+n_2+1)^{\ds -\alpha}e^{\ds i(n_1\theta_1+n_2\theta_2)}.
\label{sum2Example}
\end{equation}

It follows that 
$$f(s/(1-s),t/(1-t))=\frac{(1-s)^{\ds \alpha}(1-t)^{\ds \alpha}}{(1-st)^{\ds \alpha}},
$$
so that the functional equation \eqref{Func2eq} becomes
\begin{equation}
u(\frac{1}{2-s},\frac{1}{2-t})e^{\ds i(\theta_1+\theta_2)}
-u(\frac{1}{2-s},t)e^{\ds i\theta_1}
-u(s,\frac{1}{2-t})e^{\ds i\theta_2}
+u(s,t)
=f(\frac{s}{1-s},\frac{t}{1-t}).
\end{equation}
As a numerical test, we consider the case \(\theta_1=\theta_2=1\) and \(\alpha=2\). We note that this series is only conditionally convergent. The double series is understood in the sense of rectangular partial sums. We emphasize that this example lies outside the absolute-convergence assumptions made above and is not covered by the preceding smooth-existence
result.
We include it to illustrate the method's performance in the conditionally convergent case.
The series is

\begin{equation}
S=
\sum_{\ds n_1,n_2=0}^{\ds\infty}
\frac{e^{\ds i(n_1+n_2)}}
     {(n_1+n_2+1)^2}.
\label{2Dexample}
\end{equation}
Here
\[
q(n_1,n_2)=n_1+n_2,
\qquad
f(n_1,n_2)=(n_1+n_2+1)^{-2}.
\]
Introducing
\[
n_1=\frac{s}{1-s},
\qquad
n_2=\frac{t}{1-t},
\]
we obtain
\[
n_1+n_2+1
=
\frac{1-st}{(1-s)(1-t)}.
\]
Consequently, the transformed right-hand side is
\begin{equation}
g(s,t)
=
\left[
\frac{(1-s)(1-t)}
     {1-st}
\right]^2.
\label{2Drhs}
\end{equation}
The function $g$ is defined to be zero on the faces $s=1$ and $t=1$,
by continuity.

Let
\[
T(x)=\frac{1}{2-x},
\qquad
z=e^{\ds i}.
\]
For the linear phase considered here, the corresponding functional
equation takes the form
\begin{equation}
z^2 u(T(s),T(t))
-z u(T(s),t)
-z u(s,T(t))
+u(s,t)
=
g(s,t).
\label{2Dfunctional}
\end{equation}
Equivalently, defining
\[
(L_1u)(s,t)=z\,u(T(s),t)-u(s,t),
\]
and
\[
(L_2u)(s,t)=z\,u(s,T(t))-u(s,t),
\]
equation \eqref{2Dfunctional} can be written in the factorized form
\[
L_1L_2u=g.
\]
Since $d=2$ and $q(0,0)=0$, the required series value is
\begin{equation}
S=u(0,0).
\label{2DSvalue}
\end{equation}

For the numerical computations, we employ the Tal--Ezer mapped Chebyshev basis (see \cite{Tal-Ezer2014}), associated with the Kosloff--Tal--Ezer mapping introduced in \cite{KTE}. Specifically, we use

$$
\phi_k(x)
=
T_k\left(
\frac{\sin\bigl(p(2x-1)\bigr)}{\sin p}
\right),
\qquad k=0,\ldots,n-1,
$$

where

$$
p=2\tan^{-1}\left(\epsilon^{1/n}\right),
\qquad
\epsilon=\operatorname{eps}\approx 2.22\times10^{-16}.
$$

The two-dimensional approximation is then taken in tensor-product form,

$$
u_n(s,t)
=
\sum_{j,k=0}^{n-1}
a_{jk}\phi_j(s)\phi_k(t).
$$

Rather than using a square collocation system, we employ an
oversampled least-squares formulation. We take $m=2n$
equidistant points in each coordinate direction. Hence, the $n^2$
coefficients are determined from $m^2=4n^2$ equations in the least-squares sense.

Let
\[
D=z\Phi_T-\Phi,
\]
where
\[
\Phi_{jk}=\phi_{k-1}(x_j),
\qquad
(\Phi_T)_{jk}=\phi_{k-1}(T(x_j)).
\]
The full least-squares matrix has the tensor-product form
\[
A_{\rm LS}=D\otimes D.
\]
This matrix need not be formed explicitly. Instead, writing the
coefficients as an $n\times n$ matrix $C$, the least-squares problem
can be expressed as
\[
\min_C
\left\|
DCD^T-G
\right\|_F,
\]
where $G$ contains the values of $g$ on the $m\times m$ grid.
The problem can therefore be solved by two successive one-dimensional
least-squares operations. Moreover,
\begin{equation}
\kappa(A_{\rm LS})=\kappa(D)^2,
\label{2Dcond}
\end{equation}
where $\kappa(\cdot)$ denotes the spectral condition number.

It can be shown that the series \eqref{2Dexample} has the value
\begin{equation}
S=
-\frac{\log(1-e^i)}{e^i}.
\label{2Dexact}
\end{equation}
Numerically,
\begin{equation}
S=
0.9237472755256664
+
0.5431955295344024\,i.
\label{2Dreference}
\end{equation}

Table~\ref{tab:2DLS} presents the results obtained with the
least-squares Tal--Ezer approximation. Here
\[
E_n=|S_n-S|,
\qquad
S_n=u_n(0,0).
\]

\begin{table}[htbp]
\centering
\caption{Two-dimensional least-squares approximation of
\eqref{2Dexample}, using $m=2n$ points in each coordinate direction.}
\label{tab:2DLS}

\renewcommand{\arraystretch}{1.3}
\begin{tabular}{c c c c}
\hline
\rule{0pt}{3.0ex}
$n$ & $\operatorname{cond}(A_{\rm LS})$
    & $S_n$ & $E_n$ \\[1.2ex]
\hline

5
& $2.9140\times10^{1}$
& $0.919279773090+0.549241212797\,i$
& $7.517\times10^{-3}$ \\

10
& $1.8777\times10^{2}$
& $0.923672782519+0.543153332138\,i$
& $8.561\times10^{-5}$ \\

20
& $9.3031\times10^{2}$
& $0.923747266083+0.543195627069\,i$
& $9.799\times10^{-8}$ \\

40
& $2.7227\times10^{4}$
& $0.923747275523+0.543195529537\,i$
& $3.943\times10^{-12}$ \\

80
& $4.4898\times10^{6}$
& $0.923747275526+0.543195529534\,i$
& $3.568\times10^{-15}$ \\
\hline
\end{tabular}
\end{table}

The results demonstrate rapid convergence of the two-dimensional
least-squares approximation. Although the condition number of the tensor-product
least-squares system increases with $n$, the approximation of the
quantity $u(0,0)$ remains highly accurate.
We use the Tal--Ezer transformed Chebyshev basis because it provides a flexible polynomial-like approximation while improving the numerical behavior of the basis at high orders. In addition, instead of enforcing the functional equation at exactly as many points as there are unknown coefficients, we use an oversampled least-squares formulation. The additional equations significantly reduce the effect of the severe ill-conditioning observed in square tensor-product collocation and lead to much more stable and accurate approximations of the required value of $u$.

\subsection{$d$-dimensional series}\hfill

For \(d\)-dimensional series, we derive the corresponding functional equation below.

Let
\begin{equation}
S=\sum_{\ds {\mathbf{n}\in \mathbb{N}_0^d}} f(\mathbf{n})e^{\ds iq(\mathbf{n})},
\label{sumdfeq}
\end{equation}
where $\mathbf{n}=(n_1,n_2,\dots ,n_d)$ and we assume that the series is absolutely convergent.

We search for a $d$-dimensional sequence $\{p(\mathbf{n})\}$, $\mathbf{n}\in\mathbb{N}_0^d$, such that
\[ p(\mathbf n)e^{\ds iq(\mathbf n)} = (-1)^d \sum_{k_1=n_1}^{\infty}\cdots \sum_{k_d=n_d}^{\infty} f(\mathbf k)e^{\ds iq(\mathbf k)}. \] 
It then follows that 
\begin{equation}
\Delta_1\Delta_2\dots \Delta_d \left[ p(\mathbf{n})e^{\ds iq(\mathbf{n})} \right] =
f(\mathbf{n})e^{\ds iq(\mathbf{n})},\ \ \ \mathbf{n}\in \mathbb{N}_0^d.
\label{sumdpeq}
\end{equation}
Moreover, absolute convergence ensures the required boundary limits. Consequently, 
\[ S = (-1)^d p(\mathbf 0)e^{\ds iq(\mathbf 0)}. \] 

We look for an auxiliary function $u(\mathbf{s})$ where
$\mathbf{s}=(s_1,s_2,\dots ,s_d)$ in $[0,1]^d$, such that the $d$-dimensional sequence satisfies
$$
p(\mathbf{n})\equiv u(\frac{n_1}{n_1+1},\frac{n_2}{n_2+1},\dots, \frac{n_d}{n_d+1}).
$$

Using the expression
\begin{equation}
\Delta_1\Delta_2\cdots\Delta_d F(\mathbf n)
=
\sum_{\boldsymbol{\ds \varepsilon}\in\{0,1\}^d}
(-1)^{\ds d-|\boldsymbol{\varepsilon}|}
F(\mathbf n+\boldsymbol{\varepsilon}),
\label{Deltad}
\end{equation}
\[
\boldsymbol{\varepsilon}=(\varepsilon_1,\dots,\varepsilon_d),\ \ \
|\boldsymbol{\varepsilon}|
=
\varepsilon_1+\cdots+\varepsilon_d,
\]
and proceeding analogously to the two-dimensional case, we derive the following functional equation for $u(\mathbf{s})$:

\begin{equation}
\sum_{\boldsymbol{\ds \varepsilon}\in\{0,1\}^d}
(-1)^{\ds d-|\boldsymbol{\varepsilon}|}
u((\mathbf{1}-\boldsymbol{\varepsilon})\odot\mathbf{s}
+
\boldsymbol{\varepsilon}\odot T(\mathbf{s}))
\frac{Q((\mathbf{1}-\boldsymbol{\varepsilon})\odot B(\mathbf{s})
+
\boldsymbol{\varepsilon}\odot A(\mathbf{s}))}{Q(B(\mathbf{s}))}=f(B(\mathbf{s})), \ \mathbf{s}\in [0,1)^d.
\label{FEqud}
\end{equation}
The following notation and terms are used in \eqref{FEqud}:

$$\mathbf{1}=(1,1,\dots,1),\ \ T(\mathbf{s})=\left(\frac{1}{2-s_1},\frac{1}{2-s_2},\dots,\frac{1}{2-s_d}\right),\ \ Q(\mathbf{x})= e^{\ds iq(\mathbf{x})},$$
$$A(\mathbf{s})=\left(\frac{1}{1-s_1},\frac{1}{1-s_2},\dots,\frac{1}{1-s_d}\right),\ \ B(\mathbf{s})=\left(\frac{s_1}{1-s_1},\frac{s_2}{1-s_2},\dots,\frac{s_d}{1-s_d}\right),$$
and the symbol $\odot$ denotes the componentwise Hadamard product.

To explain \eqref{FEqud}, note that the $j$th argument of $u$ is chosen according to the value of $\varepsilon_j$: if $\varepsilon_j=0$, it is $s_j$, whereas if $\varepsilon_j=1$, it is $1/(2-s_j)$. This choice can be written compactly as

$$
(1-\varepsilon_j)s_j+\varepsilon_j\frac{1}{2-s_j}.
$$
Using the Hadamard product, this choice can be expressed compactly and simultaneously for all arguments of $u$. An analogous construction is used for the argument of $Q$.

Let $u(\mathbf{s})$ be a solution of \eqref{FEqud}. As above, we make the essential assumption that \eqref{FEqud} admits a sufficiently smooth, slowly varying solution, which may then be approximated by collocation. It follows that
$$ S=(-1)^{\ds d} u(\mathbf{0})e^{\ds iq(\mathbf{0})}.$$

As in the two-dimensional case, suppose that
\[
q(\mathbf n)=\sum_{j=1}^d\theta_jn_j,
\qquad e^{\ds i\theta_j}\ne1,\quad j=1,\ldots,d.
\]
If the transformed right-hand side
\(g(\mathbf s)\equiv f(B(\mathbf s))\in C^\infty([0,1]^d)\) vanishes on every face
\(s_j=1\), then successive applications of
Theorem~\ref{ThmSmoothSolution}, treating the remaining variables
as parameters, yield a unique smooth solution
\(u\in C^\infty([0,1]^d)\) that vanishes on each of these faces.

\subsection{A three-dimensional example}

We next consider the three-dimensional oscillatory series
\begin{equation}
S=
\sum_{\ds n_1,n_2,n_3=0}^{\ds \infty}
\frac{e^{\ds i(n_1+n_2+n_3)}}
     {(n_1+n_2+n_3+1)^3}.
\label{3Dexample}
\end{equation}
As in the two-dimensional example, this series is only conditionally
convergent. The multiple series is understood in the sense of rectangular
partial sums.

Here
\[
q(n_1,n_2,n_3)=n_1+n_2+n_3,
\qquad
f(n_1,n_2,n_3)
=(n_1+n_2+n_3+1)^{-3}.
\]
Introducing
\[
n_1=\frac{s}{1-s},\qquad
n_2=\frac{t}{1-t},\qquad
n_3=\frac{r}{1-r},
\]
we obtain
\[
n_1+n_2+n_3+1
=
\frac{1-st-sr-tr+2str}
     {(1-s)(1-t)(1-r)}.
\]
Consequently, the transformed right-hand side is
\begin{equation}
g(s,t,r)
=
\left[
\frac{(1-s)(1-t)(1-r)}
     {1-st-sr-tr+2str}
\right]^3 .
\label{3Drhs}
\end{equation}
The function $g$ is defined to be zero on the faces $s=1$, $t=1$,
and $r=1$, by continuity.

Let
\[
T(x)=\frac{1}{2-x},
\qquad z=e^{\ds i}.
\]
For the linear phase considered here, the three-dimensional functional
equation corresponding to \eqref{3Dexample} takes the form
\begin{align}
&z^3 u(T(s),T(t),T(r))
\nonumber\\
&\quad
-z^2\bigl[
u(T(s),T(t),r)
+u(T(s),t,T(r))
+u(s,T(t),T(r))
\bigr]
\nonumber\\
&\quad
+z\bigl[
u(T(s),t,r)
+u(s,T(t),r)
+u(s,t,T(r))
\bigr]
-u(s,t,r)
\nonumber\\
&\hspace{5cm}=g(s,t,r).
\label{3Dfunctional}
\end{align}
Equivalently, if
\[
(L_1u)(s,t,r)=z\,u(T(s),t,r)-u(s,t,r),
\]
with analogous definitions of $L_2$ and $L_3$, then
\[
L_1L_2L_3u=g.
\]
Since $d=3$ and $q(0,0,0)=0$, the required series value is
\begin{equation}
S=-u(0,0,0).
\label{3DSvalue}
\end{equation}

For the numerical approximation, we again employ the Tal--Ezer mapped Chebyshev basis introduced in Section~\ref{2Dnumerical}, now in three-dimensional tensor-product form:

$$
u_n(s,t,r)
=
\sum_{j,k,\ell=0}^{n-1}
a_{jk\ell}\,
\phi_j(s)\phi_k(t)\phi_\ell(r).
$$

Rather than using a square collocation system, we employ an
oversampled least-squares formulation.  We take
\[
m=2n
\]
equidistant points in each coordinate direction.  Hence, the
$n^3$ coefficients are determined from $m^3=8n^3$ equations in the
least-squares sense.

If
\[
D=z\Phi_T-\Phi,
\]
where
\[
\Phi_{jk}=\phi_{k-1}(x_j),
\qquad
(\Phi_T)_{jk}=\phi_{k-1}(T(x_j)),
\]
then the full least-squares matrix has the tensor-product form
\[
A_{\rm LS}=D\otimes D\otimes D.
\]
This matrix need not be formed explicitly.  The least-squares
problem can instead be solved by three successive one-dimensional
least-squares operations.  Moreover,
\begin{equation}
\kappa(A_{\rm LS})=\kappa(D)^3,
\label{3Dcond}
\end{equation}
where $\kappa(\cdot)$ denotes the spectral condition number.

It can be shown that the three-dimensional series \eqref{3Dexample} has the value
\begin{equation}
S=
\frac{
-\log(1-e^i)+\operatorname{Li}_2(e^i)
}
{2e^i},
\label{3Dexact}
\end{equation}
where $\operatorname{Li}_2$ denotes the dilogarithm function.

Numerically,
\begin{equation}
S
=
0.9760484167789722
+
0.4091437417344512\,i.
\label{3Dreference}
\end{equation}

Table~\ref{3Dtable} presents the results obtained with the
least-squares Tal–Ezer approximation.  Here
\[
E_n=|S_n-S|,
\qquad
S_n=-u_n(0,0,0).
\]

\renewcommand{\arraystretch}{1.35}

\begin{table}[htbp]
\centering
\caption{Three-dimensional least-squares approximation of
\eqref{3Dexample}, using $m=2n$ points in each coordinate direction.}
\label{3Dtable}
\begin{tabular}{c c c c}
\hline
\rule{0pt}{3.2ex}
$n$ & $\operatorname{cond}(A_{\rm LS})$
    & $S_n$ & $E_n$ \\[1.5ex]
\hline

 5  & $1.5730\times10^{2}$
    & $0.969680611861+0.403985346576\,i$
    & $8.195\times10^{-3}$ \\

10  & $2.5730\times10^{3}$
    & $0.976137352907+0.409054677512\,i$
    & $1.259\times10^{-4}$ \\
20  & $2.8375\times10^{4}$
    & $0.976048254303+0.409143563607\,i$
    & $2.411\times10^{-7}$ \\
40  & $4.4927\times10^{6}$
    & $0.976048416803+0.409143741738\,i$
    & $2.474\times10^{-11}$ 
    \\
80  &  $9.5134\times10^9$
    &
    $0.976048416779 +0.409143741734\,i$
    & $1.755\times10^{-15}$
    \\
\hline
\end{tabular}
\end{table}

The results demonstrate rapid convergence of the three-dimensional
least-squares approximation. Although the condition number of the
full tensor-product system increases with $n$, the approximation of the
single quantity $u(0,0,0)$ remains highly accurate.

\section{Further challenges and research directions}

The summation of one-dimensional series can be carried out very effectively using the \(d\)-transformation introduced in \citep{LevinSidi1981} and discussed in detail in Sidi's book \citep{Sidi2003}. We believe that the main potential of the series collocation method presented in this paper lies in its application to multidimensional series.

Further research is needed to broaden the class of series for which the collocation method is effective. A central theoretical question is to identify conditions on the functions \(f\) and \(q\) that ensure the existence of a smooth, non-oscillatory solution to the associated functional equation.

Another important challenge is to make the method applicable to series involving special functions, as done in \citep{Levin1996} and \citep{Olver2007} for the collocation method for highly oscillatory functions on finite intervals.

An important computational issue concerns the choice of basis functions and collocation points for solving the functional equations developed here. In this regard, the extensive experience accumulated with collocation methods for finite oscillatory integration can provide useful guidance. However, since the present problem involves a functional equation rather than a differential equation, alternative solution strategies may also be worth exploring.

As discussed in Section~\ref{Rational}, applying the collocation method to a power series yields rational approximations to the corresponding function. Unlike Padé approximants, which are determined by a finite initial segment of the coefficient sequence, the present construction exploits a functional representation of the coefficient sequence over its full range. These new approximations, in both one and several dimensions, may warrant further investigation, particularly concerning their approximation properties and the distribution of their singularities.

\bibliographystyle{plainnat}
\bibliography{references}

\end{document}